\documentclass[10pt,twoside]{amsart} 
\usepackage{amsmath, amsthm, amscd, amsfonts, amssymb, graphicx, color}
\usepackage[bookmarksnumbered, colorlinks, plainpages]{hyperref}

\newtheorem{theorem}{Theorem}[section]

\newtheorem{proposition}[theorem]{Proposition}
\newtheorem{corollary}[theorem]{Corollary}

\theoremstyle{definition}
\newtheorem{definition}[theorem]{Definition}

\newtheorem{remark}[theorem]{Remark}

\numberwithin{equation}{section}

\def\<{\langle}
\def\>{\rangle}

\long\def\alert#1{\smallskip{\hskip\parindent\vrule%
\vbox{\advance\hsize-2\parindent\hrule\smallskip\parindent.4\parindent%
\narrower\noindent#1\smallskip\hrule}\vrule\hfill}\smallskip}

\begin{document}
\title[$\phi$-classical prime submodules]{$\phi$-classical prime submodules}
\author[Mostafanasab, Sevim, Babaei and Tekir]{Hojjat Mostafanasab$^*$, Esra Sengelen Sevim,\\ Sakineh Babaei and \"{U}nsal Tekir}

\thanks{$^*$Corresponding author}

\begin{abstract}
In this paper, all rings are commutative with nonzero identity.
Let $M$ be an $R$-module. A proper submodule $N$ of $M$ 
is called a {\it classical prime submodule}, if for each  
$m\in M$ and elements $a,b\in R$, $abm\in N$ implies
that $am\in N$ or $bm\in N$. 
Let $\phi :S(M)\rightarrow S(M)\cup \{\emptyset \}$ be a function
where $S(M)$ is the set of all submodules of $M$.
We introduce the concept of ``$\phi$-classical prime submodules''. A proper submodule
$N$ of $M$ is a {\it $\phi$-classical prime submodule} if whenever
$a,b\in R$ and $m\in M$ with $abm\in N\backslash\phi(N)$,
then $am\in N$ or $bm\in N$.
\end{abstract}

\maketitle

\section{\protect\bigskip Introduction}

Throughout this paper all rings are commutative with nonzero identity and 
all modules are considered to be unitary. 
Anderson and Smith \cite{AS} said that a proper ideal $I$ of a ring $R$ is {\it weakly prime}
if whenever $a,b\in R$ with $0\neq ab\in I$, then $a\in I$ or $b\in I$.
In \cite{BS}, Bhatwadekar and Sharma defined a proper ideal $I$ of an integral domain $R$ to be {\it almost prime (resp. $n$-almost prime)} if for $a, b\in R$ with $ab\in I\backslash I^2$, (resp. $ab\in I\backslash I^n, n\geq 3$) either $a\in I$ or $b\in I$. This definition can obviously be made for any commutative ring $R$.  Later, Anderson and Batanieh \cite{AB} gave a generalization of prime ideals which covers all the above mentioned definitions. Let $\phi:\mathfrak{J}(R)\rightarrow \mathfrak{J}(R)\cup\{\emptyset\}$ be a function. A proper ideal $I$ of $R$ is said to be $\phi$-{\it prime} if for $a, b\in R$ with $ab\in I \backslash\phi(I)$, $a\in I$ or $b\in I$.
Several authors have extended the notion
of prime ideal to modules, see, for example \cite{D,L,MM}. Let $M$ be a module
over a commutative ring $R$. A proper submodule $N$ of $M$ is called \textit{prime} 
if for $a\in R$ and $m\in M$, $am\in N$ implies that $m\in N$ or $a\in (N:_{R}M)=\{r\in R\mid rM\subseteq N\}$. 
Weakly prime submodules were introduced by Ebrahimi Atani and Farzalipour in \cite{E1}. A proper
submodule $N$ of $M$ is called \textit{weakly prime} if for $a\in R$ and $m\in M$
with $0\neq am\in N$, either $m\in N$ or $a\in (N:_{R}M)$.
Zamani \cite{Z} introduced the concept of $\phi$-prime submodules. Let 
$\phi:S(M)\rightarrow S(M)\cup \{\emptyset \}$ be a function where $S(M)$ is the
set of all submodules of $M$. A proper submodule $N$ of an $R$-module $M$ is
called {\it $\phi$-prime} if $a\in R$ and $m\in M$ with 
$am\in N\backslash\phi (N)$, then $m\in N$ or $a\in (N:_{R}M)$.
He defined the map $\phi_{\alpha}:S(M)\rightarrow S(M)\cup\{\emptyset\}$ as follows:
\begin{itemize}
\item[(1)] $\phi_{\emptyset}$ : $\phi(N) = \emptyset$ defines prime submodules.
\item[(2)] $\phi_{0}$ : $\phi(N)=\{0\}$ defines weakly prime submodules.
\item[(3)] $\phi_{2}$ : $\phi(N)=(N:_RM)N$ defines almost prime submodules.
\item[(4)] $\phi_{n} (n\geq 2)$ : $\phi(I)=(N:_RM)^{n-1}N$ defines $n$-almost prime submodules.
\item[(5)] $\phi_{\omega}$ : $\phi(N) = \cap_{n=1}^{\infty} (N:_RM)^nN$ defines $\omega$-prime submodules.
\item[(6)] $\phi_{1}$ : $\phi(N) = N$ defines any submodules.
\end{itemize}
Also, Moradi and Azizi \cite{MA} investigated the notion of $n$-almost prime submodules.
A proper submodule $N$ of $M$ is called a \textit{classical prime submodule}, if for
each $m\in M$ and $a,b\in R$, $abm\in N$ implies that $am\in N$ or $%
bm\in N$. This notion of classical prime submodules has been extensively
studied by Behboodi in \cite{B1,B2} (see also, \cite{BK}, in which, the
notion of classical prime submodules is named \textquotedblleft weakly prime submodules\textquotedblright).
For more information on classical prime submodules, the reader is referred to \cite{A1,A,BSh}.
In \cite{MTO}, Mostafanasab et. al. said that a proper submodule $N$ of an $R$-module $M$ is called a
{\it weakly classical prime submodule} if whenever
$a,b\in R$ and $m\in M$ with $0\neq abm\in N$,
then $am\in N$ or $bm\in N$.

Let $\phi :S(M)\rightarrow S(M)\cup \{\emptyset \}$ be a function
where $S(M)$ is the set of all submodules of $M$. Let $N$ be a proper
submodule of $M$. Then we say that $N$ is a $\phi $-\textit{classical prime submodule} of $M$ if
whenever $a,b\in R$ and $m\in M$ with $abm\in N\backslash\phi(N)$, then 
$am\in N$ or $bm\in N$. Clearly, every classical prime submodule is a $\phi$-classical prime submodule.
We defined the map $\phi_{\alpha}:S(M)\rightarrow S(M)\cup\{\emptyset\}$ as follows:
\begin{itemize}
\item[(1)] $\phi_{\emptyset}$ : $\phi(N) = \emptyset$ defines classical prime submodules.
\item[(2)] $\phi_{0}$ : $\phi(N)=\{0\}$ defines weakly classical prime submodules.
\item[(3)] $\phi_{2}$ : $\phi(N)=(N:_RM)N$ defines almost classical prime submodules.
\item[(4)] $\phi_{n} (n\geq 2)$ : $\phi(I)=(N:_RM)^{n-1}N$ defines $n$-almost classical prime submodules.
\item[(5)] $\phi_{\omega}$ : $\phi(N) = \cap_{n=1}^{\infty} (N:_RM)^nN$ defines $\omega$-classical prime submodules.
\item[(6)] $\phi_{1}$ : $\phi(N) = N$ defines any submodules.
\end{itemize}
Throughout this paper $\phi :S(M)\rightarrow S(M)\cup\{\emptyset \}$ denotes a function. 
Since $N\backslash\phi (N)=N\backslash(N\cap \phi (N))$, for any
submodule $N$ of $M$, without loss of generality we may assume
that $\phi (N)\subseteq N$. For any two functions $\psi _{1},$ $\psi
_{2}:S(M)\rightarrow S(M)\cup \{\emptyset \}$, we say $\psi _{1}\leq \psi
_{2}$ if $\psi _{1}(N)\subseteq \psi _{2}(N)$ for each $N\in S(M)$. 
Thus clearly we have the following order: $\phi _{\emptyset }\leq \phi _{0}\leq
\phi _{\omega }\leq ...\leq \phi _{n+1}\leq \phi _{n}\leq ...\leq \phi
_{2}\leq \phi _{1}$. Whenever $\psi_1\leq \psi_2$, any $\psi_1$-classical prime submodule is $\psi_2$-classical prime. 

An $R$-module $M$ is called
a \textit{multiplication module} if every submodule $N$ of $M$ has the form $%
IM$ for some ideal $I$ of $R$, see \cite{ES}. Note that, since $I\subseteq (N:_{R}M)$ then $%
N=IM\subseteq (N:_{R}M)M\subseteq N$. So that $N=(N:_{R}M)M$.
Let $N$ and $K$ be submodules of a multiplication 
$R$-module $M$ with $N=I_{1}M$ and $K=I_{2}M$ for some ideals $I_{1}$ and $%
I_{2}$ of $R$. The product of $N$ and $K$ denoted by $NK$ is defined by $%
NK=I_{1}I_{2}M$. Then by \cite[Theorem 3.4]{Am}, the product of $N$ and $K$
is independent of presentations of $N$ and $K$. Moreover, for $m,m'\in M$, by 
$mm'$, we mean the product of $Rm$ and $Rm'$. Clearly, $NK$ is a submodule of $%
M$ and $NK\subseteq N\cap K$ (see \cite{Am}). Let $N$ be a proper submodule
of a nonzero $R$-module $M$. Then the $M$-radical of $N$, denoted by $M$-$%
\mathrm{rad}(N)$, is defined to be the intersection of all prime submodules
of $M$ containing $N$. If $M$ has no prime submodule containing $N$, then we
say $M$-$\mathrm{rad}(N)=M$. It is shown in \textrm{\cite[Theorem 2.12]{ES}}
that if $N$ is a proper submodule of a multiplication $R$-module $M$, then $%
M $-$\mathrm{rad}(N)=\sqrt{(N:_{R}M)}M$. 

In \cite{Q}, Quartararo et. al. said that a commutative ring $R$ is a $u$%
-ring provided $R$ has the property that an ideal contained in a finite
union of ideals must be contained in one of those ideals; and a $um$-ring is
a ring $R$ with the property that an $R$-module which is equal to a finite
union of submodules must be equal to one of them. 
They show that every B$\acute{\rm e}$zout ring is a $u$-ring. Moreover, they proved that 
every Pr\"{u}fer domain is a $u$-domain. Also, any ring which contains an infinite field
as a subring is a $u$-ring, \cite[Exercise 3.63]{Sh}.

Let $M$ be an $R$-module and $\phi :S(M)\rightarrow S(M)\cup\{\emptyset \}$ be a function.
It is shown (Theorem \ref{main}) that $N$ is a $\phi$-classical prime submodule of
$M$ if and only if for every ideals $I,~J$ of $R$ and $m\in M$ with $IJm\subseteq N$ and $IJm\nsubseteq\phi(N)$, either
$Im\subseteq N$ or $Jm\subseteq N$. 
It is shown (Theorem \ref{main2}) that  over a $um$-ring $R$, 
$N$ is a $\phi$-classical prime submodule of $M$ if and only if for every ideals $I,~J$ of $R$ and submodule
$L$ of $M$ with $IJL\subseteq N$ and $IJL\nsubseteq\phi(N)$, either $IL\subseteq N$ or $JL\subseteq N$.
It is proved  (Theorem \ref{T2})  that if $N$ is a $\phi$-classical prime submodule of $M$ that is not classical prime, then $(N:_{R}M)^{2}N\subseteq\phi(N)$. Let $M_{1}, M_{2}$ be $R$-modules and $N_{1}$ be a proper submodule of $M_1$. 
Suppose that $\psi _{i}:S(M_{i})\rightarrow S(M_{i})\cup\{\emptyset \}$ be functions (for $i=1,2$) and let $\phi =\psi _{1}\times\psi _{2}$. In Theorem \ref{prod1} we prove that the following conditions are equivalent:

(1) $N_{1}\times M_{2}$ is a $\phi$-classical prime submodule of $M_{1}\times M_{2}$;

(2) $N_1$ is a $\phi$-classical prime submodule of $M_1$ and for each $r,s\in R$ and $m_1\in M_1$

\hspace{0.4cm} we have $rsm_1\in\psi_1(N_1),\ rm_1\notin N_1, \ sm_1\notin N_1\Rightarrow rs\in (\psi_2(M_2):_RM_2).$\\
Let $R=R_{1}\times R_{2}\times R_3$ be a decomposable ring and $M=M_{1}\times M_{2}\times M_3$ be an $R$-module where
$M_{i}$ is an $R_{i}$-module, for $i=1,2,3$. Suppose that $\psi _{i}:S(M_{i})\rightarrow S(M_{i})\cup\{\emptyset \}$ be functions such that $\psi(M_i)\neq M_i$ for $i=1,2,3,$ and let $\phi =\psi _{1}\times\psi _{2}\times\psi_3.$  
In Theorem \ref{product3} it is proved that if $N$ is a $\phi$-classical prime submodule of $M$, then either $N=\phi(N)$ or $N$ is a classical prime submodule of $M$.

\bigskip

\section{Properties of $\phi$-classical prime submodules}

Let $M$ be an $R$-module, $K$ be a submodule of $M$ and $\phi:S(M)\to S(M)\cup\{\emptyset\}$ 
be a function. Define $\phi_K:S(M/K)\to S(M/K)\cup\{\emptyset\}$ by $\phi_K(N/K)=(\phi(N)+K)/K$
for every $N\in S(M)$ with $N\supseteq K$ (and $\phi_K(N/K)=\emptyset$ if $\phi(N)=\emptyset$).
\begin{theorem}\label{frac}
Let $M$ be an $R$-module and $K\subseteq N$ be proper submodules of $M$. Suppose that $\phi :S(M)\rightarrow S(M)\cup
\{\emptyset \}$ be a function.
\begin{enumerate}
\item If $N$ is a $\phi$-classical prime submodule of $M$, then $N/K$ is a 
$\phi_K$-classical prime submodule of $M/K$. 

\item If $K\subseteq\phi(N)$ and $N/K$ is a  $\phi_K$-classical prime submodule of $M/K$, then $N$ is a $\phi$-classical prime submodule of $M$.

\item If $\phi(N)\subseteq K$ and $N$ is a $\phi$-classical prime submodule of $M$, then $N/K$ is a 
weakly classical prime submodule of $M/K$. 

\item If $\phi(K)\subseteq\phi(N)$, $K$ is a $\phi$-classical prime submodule of $M$ and $N/K$ is a weakly classical prime submodule of $M/K$, then $N$ is a $\phi$-classical prime submodule of $M$.
\end{enumerate}
\end{theorem}
\begin{proof}
(1) Let $a,b\in R$ and $m\in M$ be such that $ab(m+K)\in(N/K)\backslash\phi_K(N/K)$. It follows that 
$abm\in N\backslash\phi(N)$, that gives $am\in N$ or $bm\in N$. Therefore $a(m+K)\in N/K$ or $b(m+K)\in N/K$.

(2) Let $a,b\in R$ and $m\in M$ be such that $abm\in N\backslash\phi(N)$. 
Then $ab(m+K)\in(N/K)\backslash\phi_K(N/K)=(N/K)\backslash(\phi(N)/K)$. Hence 
$a(m+K)\in N/K$ or $b(m+K)\in N/K$, and so $am\in N$ or $bm \in N$. 

(3) Let $a,b\in R$ and $m\in M$ be such that $0\neq ab(m+K)\in(N/K)$. Hence 
$abm\in N\backslash\phi(N)$, because  $\phi(N)\subseteq K$. Thus $am\in N$ or $bm\in N$. Therefore $a(m+K)\in N/K$ or $b(m+K)\in N/K$.

(4) Let $a,b\in R$ and $m\in M$ be such that $abm\in N\backslash\phi(N)$. 
Note that $\phi(K)\subseteq\phi(N)$ implies that $abm\notin\phi(K)$.
If $abm\in K$, then $am\in K\subseteq N$ or $bm\in K\subseteq N$, since $K$ is $\phi$-classical prime. 
Now, assume that $abm\notin K$. So $0\neq ab(m+K)\in N/K$. Therefore, since
$N/K$ is a weakly classical prime submodule of $M/K$, either
$a(m+K)\in N/K$ or $b(m+K)\in N/K$. Thus $am\in N$ or $bm \in N$. 
\end{proof}

\begin{corollary}
Let $N$ be a proper submodule of $M$ and $\phi :S(M)\rightarrow S(M)\cup\{\emptyset \}$ be a function.
Then $N$ is a $\phi$-classical prime submodule of $M$ if and only if $N / \phi(N)$ is a weakly classical prime submodule of $M/\phi(N)$.
\end{corollary}

\begin{theorem}\label{colon}
Let $M$ be an $R$-module and $N$ be a proper submodule of $M$. Suppose that $\phi :S(M)\rightarrow S(M)\cup
\{\emptyset \}$ and $\psi :\mathfrak{J}(R)\rightarrow\mathfrak{J}(R)\cup\{\emptyset \}$ be two functions.
\begin{enumerate}
\item If $N$ is a $\phi$-classical prime submodule of $M$, then $\left( N:_Rm\right) $ is a 
$\psi$-prime ideal of $R$ for every $m\in M\backslash N$ with $(\phi(N):_Rm)\subseteq\psi((N:_Rm))$. 

\item If $\left( N:_Rm\right) $ is a $\psi$-prime ideal of $R$ for every $m\in M\backslash N$
with $\psi((N:_Rm))\subseteq(\phi(N):_Rm)$, then $N$ 
is a $\phi$-classical prime submodule of $M$.
\end{enumerate}
\end{theorem}

\begin{proof}
(1) Suppose that $N$ is a $\phi$-classical prime submodule. Let $m\in M\backslash N$ with $(\phi(N):_Rm)\subseteq\psi((N:_Rm))$, 
and $ab\in \left( N:_Rm\right)\backslash\psi(( N:_Rm))$ for some $a,b\in R$. Then $abm\in N\backslash\phi(N)$. So $am\in N$ or $bm\in N$.  Hence $a\in \left( N:_Rm\right)$ or $b\in \left( N:_Rm\right) $. Consequently $\left( N:_Rm\right) $ is a 
$\psi$-prime ideal of $R$.

(2) Assume that $\left( N:_Rm\right) $ is a $\psi$-prime ideal of $R$ for every $m\in M\backslash N$
with $\psi((N:_Rm))\subseteq(\phi(N):_Rm)$. Let $abm\in N\backslash\phi(N)$ 
for some $m\in M$ and $a,b\in R$. If $m\in N$, then we are done. So
we assume that $m\notin N$. Hence $ab\in(N:_Rm)\backslash\psi((N:_Rm))$ implies that either $a\in (N:_Rm)$
or $b\in (N:_Rm)$. Therefore either $am\in N$ or $bm\in N$, and so $N$ is $\phi$-classical prime.
\end{proof}

\begin{corollary}
Let $R$ be a ring and $\psi :\mathfrak{J}(R)\rightarrow\mathfrak{J}(R)\cup\{\emptyset \}$ be a function. Then
$I$ is a $\psi$-prime ideal of $R$ if and only if
$_{R}I$ is a $\psi$-classical prime submodule of $_{R}R$. 
\end{corollary}
\begin{proof}
First, note that $_{R}I$ is a 
$\psi$-prime submodule of $_{R}R$ if and only if $I$ is a $\psi$-prime ideal of $R$. Now, apply part (1) of Proposition \ref{abs-class}.
Conversely, let $_{R}I$ be a $\psi$-classical prime submodule of $_{R}R$. 
Notice that $(\psi(I):_R1)=\psi((I:_R1))=\psi(I)$.
Then by Theorem \ref{colon}(1),
$(I:_R1)=I$ is a $\psi$-prime ideal of $R$. 
\end{proof}

Darani and Soheilnia \cite{YS} generalized the concept of prime submodules of a module over a commutative ring as
follows: Let $N$ be a proper submodule of an $R$-module $M$. Then $N$ is said
to be a \textit{2-absorbing submodule of} $M$ if whenever $a,b\in R$ and $m\in M$ with $abm\in N$, 
then $am\in N$ or $bm\in N$ or $ab\in (N:_{R}M)$.
Let $N$ be a proper submodule of an $R$-module $M$. Then $N$ is said
to be a \textit{$\phi$-2-absorbing submodule of} $M$ if whenever $a,b\in R$ and $m\in M$ with $abm\in N\backslash\phi(N)$, 
then $am\in N$ or $bm\in N$ or $ab\in (N:_{R}M)$, see \cite{EN}.
\begin{proposition}\label{abs-class}
Let $N$ be a proper submodule of an $R$-module $M$. Suppose that $\phi :S(M)\rightarrow S(M)\cup
\{\emptyset \}$ and $\psi :\mathfrak{J}(R)\rightarrow\mathfrak{J}(R)\cup\{\emptyset \}$ be two functions.
\begin{enumerate}
\item If $N$ is a $\phi$-prime submodule of $M$, then $N$ is a $\phi$-classical prime submodule of $M$.

\item If $N$ is a $\phi$-classical prime submodule of $M$, then $N$ is a $\phi$-2-absorbing submodule of $M$. The 
converse holds if in addition $(N:_RM)$ is a $\psi$-prime ideal of $R$ and $\psi((N:_RM))\subseteq(\phi(N):_Rm)$.
\end{enumerate}
\end{proposition}
\begin{proof}
(1) Assume that $N$ is a $\phi$-prime submodule of $M$. Let $a,b\in R$ and $m\in M$
such that $abm\in N\backslash\phi(N)$. Therefore either $bm\in N$ or $a\in(N:_RM)$.
The first case leads us to the claim. In the second case we have that $am\in N$.
Consequently $N$ is a $\phi$-classical prime submodule.

(2) It is evident that if $N$ is $\phi$-classical prime, then it is $\phi$-2-absorbing. 
Assume that $N$ is a $\phi$-2-absorbing submodule of $M$ and $(N:_RM)$ is a $\psi$-prime ideal of $R$. 
Let $abm\in N\backslash\phi(N)$ for some $a,b\in R$ and $m\in M$ such that neither $am\in N$ nor $bm\in N$. 
Then $ab\in(N:_RM)$. Since $abm\notin\phi(N)$, then 
$\psi((N:_RM))\subseteq(\phi(N):_Rm)$ implies that $ab\notin\psi((N:_RM))$.
Therefore, either $a\in(N:_RM)$ or $b\in(N:_RM)$. This contradiction 
shows that $N$ is $\phi$-classical prime.
\end{proof}

\begin{definition}
Let $N$ be a proper submodule of a multiplication $R$-module $M$ and $n\geq 2$. Then
$N$ is said to be $n$-{\it potent classical prime} if whenever $a,b\in R$ and $m\in M$ with $abm\in N^n$, then $am\in N$ or $bm\in N$.
\end{definition}

\begin{proposition}
Let $M$ be a multiplication $R$-module. If $N$ is an $n$-almost classical prime submodule of $M$ for some $n\geq 2$ and $N$ is
a $k$-potent classical prime for some $k\leq n$, then $N$ is a classical prime submodule of $M$.
\end{proposition}

\begin{proof}
Assume that $N$ is an $n$-almost classical prime submodule of $M$. Let $abm\in N$ for some $a,b\in R$ and $m\in M$. If $abm\not\in N^k$, then $abm\not\in N^n$. In this case, we are done since $N$ is an $n$-almost classical prime submodule. So assume that $abm\in N^k$. Hence we get $am\in N$ or $bm\in N$, since $N$ is a $k$-potent classical prime submodule of $M$.
\end{proof}

\begin{proposition}
Let $M$ be a cyclic $R$-module and $\phi :S(M)\rightarrow S(M)\cup
\{\emptyset \}$ be a function. Then a proper submodule $N$ of $M$ 
is a $\phi$-prime submodule if and only if it is a $\phi$-classical prime submodule.
\end{proposition}

\begin{proof}
By Proposition \ref{abs-class}(1), the ``only if'' part holds.
Let $M=Rm$ for some $m\in M$ and $N$ be a $\phi$-classical prime submodule of $M$. Suppose that 
$rx\in N\backslash\phi(N)$ for some $r\in R$ and $x\in M$. Then there exists an element 
$s\in R$ such that $x=sm$. Therefore $rx=rsm\in N\backslash\phi(N)$ and since $N$ is a
$\phi$-classical prime submodule, $rm\in N$ or $sm\in N$. Hence $r\in (N:_RM)$
or $x\in N$. Consequently $N$ is a $\phi$-prime submodule.
\end{proof}

\begin{theorem}
Let $f:M\to M'$ be an epimorphism of $R$-modules and let $\phi:S(M)\to S(M)\cup \lbrace \emptyset \rbrace$ and $\phi':S(M')\to S(M')\cup \lbrace \emptyset \rbrace$ be functions. Then the following conditions hold:
\begin{enumerate}
\item If $N'$ is a $\phi'$-classical prime submodule of $M'$ and $\phi(f^{-1}(N'))=f^{-1}(\phi'(N'))$, then $f^{-1}(N')$ is a 
$\phi$-classical prime submodule of $M$.
\item If $N$ is a $\phi$-calssical prime submodule of $M$ cotaining $Ker(f)$ and $\phi'(f(N))=f(\phi(N))$, then $f(N)$ is a $\phi'$-classical prime submodule of $M'$.
\end{enumerate}
\end{theorem}

\begin{proof}
(1) Since $f$ is epimorphism, $f^{-1}(N')$ is a proper submodule of $M$. Let $a,b\in R$ and $m\in M$ such that $abm\in f^{-1}(N')\backslash\phi(f^{-1}(N'))$. Since $abm\in f^{-1}(N')$, $abf(m)\in N'$. Also,
$\phi(f^{-1}(N'))=f^{-1}(\phi'(N'))$ implies that $abf(m)\notin\phi'(N')$.
Thus $abf(m)\in N'\backslash\phi'(N')$.
Hence $af(m)\in N'$ or $bf(m)\in N'$ and thus $am\in f^{-1}(N')$ or $bm\in f^{-1}(N')$. So, 
we conclude that $f^{-1}(N')$ is a $\phi$-classical prime submodule of $M$.

(2) Let $a,b\in R$ and $m'\in M'$ such that $abm'\in f(N)\backslash\phi'(f(N))$. 
Since $f$ is epimorphism, there exists $m\in M$ such that $m'= f(m)$.
Therefore $f(abm)\in f(N)$ and so $abm\in N$, because ${\rm Ker}(f)\subseteq N$.
Since $\phi'(f(N))=f(\phi(N))$, then $abm\notin\phi(N)$. Hence $abm\in N\backslash\phi(N)$. 
It implies that $am\in N$ or $bm\in N$. Thus $am'\in f(N)$ or $bm'\in f(N)$. 
\end{proof}

\begin{theorem}
Let $M$ be an $R$-module and $\phi :S(M)\rightarrow S(M)\cup
\{\emptyset \}$ be a function. Suppose that $N$ is a $\phi$-classical prime submodule of $M$. Then
\begin{enumerate} 
\item[(1)] For every $a,b\in R$ and $m\in M$, $(N:_{R}abm)=(\phi(N):_{R}abm)\cup(\phi(N):_{R}am)\cup(\phi(N):_{R}bm)$;
\item[(2)] If $R$ is a $u$-ring, then for every $a,b\in R$ and $m\in M$, $(N:_{R} abm )=(\phi(N):_R abm)$ or $(N:_{R} abm )=(N:_R am)$ or $(N:_{R} abm )=(N:_R bm)$
\end{enumerate}
\end{theorem}

\begin{proof}
(1) Let $a,b\in R$ and $m\in M$. Suppose that $r\in(N:_{R}abm)$, the $ab(rm)\in N$. If $ab(rm)\in\phi(N)$, then $r \in (\phi(N):_R abm)$. Therefore we assume that $ab(rm)\notin\phi(N)$. So, either $a(rm)\in N$ or $b(rm)\in N$. Hence, either $r\in (N:_R am)$ or $r\in (N:_R bm)$. Consequently, $(N:_R abm)=(\phi(N):_R abm)\cup(N:_R am)\cup(N:_R bm)$.

(2) Apply part (1).
\end{proof}

Let $M$ be an $R$-module and $N$ a submodule of $M$. 
For every $a\in R$, $\{m\in M\mid am\in N\}$ is denoted by $(N:_M a)$. 
It is easy to see that $(N:_Ma)$ is a submodule of $M$ containing 
$N$.

In the next theorem we characterize $\phi$-classical prime submodules.
\begin{theorem}\label{main}
Let $M$ be an $R$-module, $\phi :S(M)\rightarrow S(M)\cup
\{\emptyset \}$ a function and $N$ be a proper submodule of $M$.
The following conditions are equivalent:
\begin{enumerate}
\item $N$ is $\phi$-classical prime;

\item For every $a,b\in R$, $(N:_Mab)=(\phi(N):_Mab)\cup(N:_Ma)\cup(N:_Mb)$;

\item For every $a\in R$ and $m\in M$ with $am\notin N$, $(N:_Ram)=(\phi(N):_Ram)\cup(N:_Rm)$;

\item For every $a\in R$ and $m\in M$ with $am\notin N$, $(N:_Ram)=(\phi(N):_Ram)$
or $(N:_Ram)=(N:_Rm)$;

\item For every $a\in R$ and every ideal $I$ of $R$ and $m\in M$ with $aIm\subseteq N$
and $aIm\nsubseteq\phi(N)$, either $am\in N$ or $Im\subseteq N$;

\item For every ideal $I$ of $R$ and $m\in M$ with $Im\nsubseteq N$, 
$(N:_RIm)=(\phi(N):_RIm)$ or $(N:_RIm)=(N:_Rm)$;

\item For every ideals $I,~J$ of $R$ and $m\in M$ with $IJm\subseteq N$
and $IJm\nsubseteq\phi(N)$, either $Im\subseteq N$ or $Jm\subseteq N$.
\end{enumerate}
\end{theorem}
\begin{proof}
(1)$\Rightarrow$(2) Suppose that $N$ is a $\phi$-classical prime submodule of $M$. Let $m\in
( N:_Mab)$. Then $abm\in N$. If $abm\in\phi(N)$, then $m\in(\phi(N):_Mab)$. Assume that $abm\notin\phi(N)$. 
Hence $am\in N$ or $bm\in N$. Therefore $m\in\left( N:_Ma\right)$ or $m\in( N:_Mb)$. 
Consequently, $\left(N:_Mab\right)=(\phi(N):_Mab)\cup\left(N:_Ma\right)\cup\left( N:_Mb\right)$.\newline
(2)$\Rightarrow$(3) 
Let $am\notin N$ for some $a\in R$ and $m\in M$. Assume that $x\in(N:_Ram)$.
Then $axm\in N$, and so $m\in(N:_Max)$. Since $am\notin N$, then $m\notin(N:_Ma)$.
Thus by part (2), $m\in(\phi(N):_Max)$ or $m\in(N:_Mx)$, whence $x\in(\phi(N):_Ram)$
or $x\in(N:_Rm)$. Therefore $(N:_Ram)=(\phi(N):_Ram)\cup(N:_Rm)$.\newline
(3)$\Rightarrow$(4) By the fact that if an ideal (a subgroup) is the union
of two ideals (two subgroups), then it is equal to one of them.\newline
(4)$\Rightarrow$(5) Let for some $a\in R$, an ideal $I$ of $R$, $m\in M$, we have that $aIm\subseteq N$ and $aIm\nsubseteq\phi(N)$.
Hence $I\subseteq(N:_Ram)$ and $I\nsubseteq(\phi(N):_Ram)$.
If $am\in N$, then we are done. So, assume that $am\notin N$.
Therefore by part (4) we have that $I\subseteq(N:_Rm)$, i.e., $Im\subseteq N$.\newline
(5)$\Rightarrow$(6)$\Rightarrow$(7) Have proofs similar to that of the previous implications.\newline
(7)$\Rightarrow$(1) Is trivial.
\end{proof}

\begin{remark}\label{multi}
Let $M$ be a multiplication $R$-module and $K,L$ be submodules of $M$.
Then there are ideals $I,J$ of $R$ such that $K=IM$ and $L=JM$. Thus
$KL=IJM=IL$. In particular $KM=IM=K$. Also, for any $m\in M$ we define
$Km:=KRm$. Hence $Km=IRm=Im$. 
\end{remark}

\begin{theorem}
Let $M$ be a multiplication $R$-module, $N$ be a proper submodule of $M$ and $\phi :S(M)\rightarrow S(M)\cup
\{\emptyset \}$ be a function.
Then the following conditions are equivalent:
\begin{enumerate}
\item[(1)] N is a $\phi$-classical prime submodule of M;
\item[(2)] If $N_1N_2m\subseteq N$ for some submodules $N_1$, $N_2$ of M and $m\in M$ such that $N_1N_2m\not\subseteq \phi(N)$, then either $N_1m\subseteq N$ or $N_2m\subseteq N$
\end{enumerate}
\end{theorem}
\begin{proof}
$(1)\Rightarrow (2)$. Let $N_1N_2m\subseteq N$ for some submodules $N_1$, $N_2$ of $M$ and $m\in M$. Since $M$ is multiplication, there are ideals $I_1$,$I_2$ $\in R$ such that $N_1=I_1M$ and $N_2=I_2M$. Then $N_1N_2m=I_1I_2m\subseteq N$ and  $I_1I_2m\not\subseteq \phi(N)$, so by Theorem \ref{main}, either $I_1m\subseteq N$ or $I_2m\subseteq N$. Therefore $N_1m\subseteq N$ or $N_2m\subseteq N$.
\\
$(2)\Rightarrow (1)$ Suppose that $I_1I_2m\subseteq N$ for some ideals $I_1, I_2$ of $R$ and some $m\in M$. Then it is sufficient to get $N_1=I_1M$ and $N_2=I_2M$ in $(2)$.
\end{proof}

\begin{theorem}\label{main2}
Let $R$ be a $um$-ring, $M$ be an $R$-module and $N$ be a proper submodule of $M$.
Suppose that $\phi :S(M)\rightarrow S(M)\cup\{\emptyset \}$ be a function.
The following conditions are equivalent:
\begin{enumerate}
\item $N$ is $\phi$-classical prime;

\item For every $a,b\in R$, $(N:_Mab)=(\phi(N):_Mab)$ or $(N:_Mab)=(N:_Ma)$
or $(N:_Mab)=(N:_Mb)$;

\item For every $a,b\in R$ and every submodule $L$ of $M$, $abL\subseteq N$ and $abL\nsubseteq\phi(N)$
implies that $aL\subseteq N$ or $bL\subseteq N$;

\item For every $a\in R$ and every submodule $L$ of $M$ with $aL\nsubseteq N$, $(N:_RaL)=(\phi(N):_RaL)$
or $(N:_RaL)=(N:_RL)$;

\item For every $a\in R$, every ideal $I$ of $R$ and every submodule $L$ of $M$, $aIL\subseteq N$ and $aIL\nsubseteq\phi(N)$
implies that $aL\subseteq N$ or $IL\subseteq N$;

\item For every ideal $I$ of $R$ and every submodule $L$ of $M$ with $IL\nsubseteq N$, $(N:_RIL)=(\phi(N):_RIL)$
or $(N:_RIL)=(N:_RL)$;

\item For every ideals $I,~J$ of $R$ and every submodule $L$ of $M$, $IJL\subseteq N$ and $IJL\nsubseteq\phi(N)$
implies that $IL\subseteq N$ or $JL\subseteq N$.
\end{enumerate}
\end{theorem}
\begin{proof}
(1)$\Rightarrow$(2) Assume that $N$ is a $\phi$-classical prime submodule of $M$ and $a,b\in R$.
Then by Theorem \ref{main}, $\left(N:_Mab\right)=(\phi(N):_Mab)\cup\left(N:_Ma\right)\cup\left( N:_Mb\right)$.
Since $R$ is a $um$-ring, then $(N:_Mab)=(\phi(N):_Mab)$ or $(N:_Mab)=(N:_Ma)$
or $(N:_Mab)=(N:_Mb)$.\newline
(2)$\Rightarrow$(3) 
Let $abL\subseteq N$ and $abL\nsubseteq\phi(N)$ for some $a,b\in R$ and submodule $L$ of $M$. 
Hence $L\subseteq(N:_Mab)$ and $L\nsubseteq(\phi(N):_Mab)$. Therefore part (2) implies that
$(N:_Mab)=(N:_Ma)$ or $(N:_Mab)=(N:_Mb)$. So, either $L\subseteq(N:_Ma)$ or $L\subseteq(N:_Mb)$, i.e.,
$aL\subseteq N$ or $bL\subseteq N$.\newline
(3)$\Rightarrow$(4) Let $aL\nsubseteq N$ for some $a\in R$ and submodule $L$ of $M$. Suppose that $x\in(N:_RaL)$.
Then $axL\subseteq N$. If $axL\subseteq\phi(N)$, then $x\in(\phi(N):_MaL)$. Now, assume that $axL\nsubseteq\phi(N)$.
Thus by part (3) we have that $xL\subseteq N$, because $aL\nsubseteq N$.
Therefore $x\in(N:_RL)$. Consequently $(N:_RaL)=(\phi(N):_RaL)\cup(N:_RL)$, and then
$(N:_RaL)=(\phi(N):_RaL)$ or $(N:_RaL)=(N:_RL)$.\newline
(4)$\Rightarrow$(5) Let for some $a\in R$, an ideal $I$ of $R$ and submodule $L$ of $M$, we have that $aIL\subseteq N$ and $aIL\nsubseteq\phi(N)$. Hence $I\subseteq(N:_RaL)$ and $I\nsubseteq(\phi(N):_RaL)$.
If $aL\subseteq N$, then we are done. So, assume that $aL\nsubseteq N$.
Therefore by part (4) we have that $I\subseteq(N:_RL)$, i.e., $IL\subseteq N$.\newline
(5)$\Rightarrow$(6)$\Rightarrow$(7) Similar to the previous implications.\newline
(7)$\Rightarrow$(1) Is easy.
\end{proof}

\begin{theorem}\label{col}
Let $R$ be a $um$-ring, $M$ be an $R$-module and $N$ be a proper submodule of $M$. Suppose that $\phi :S(M)\rightarrow S(M)\cup
\{\emptyset \}$ and $\psi :\mathfrak{J}(R)\rightarrow\mathfrak{J}(R)\cup\{\emptyset \}$ be two functions.
If $N$ is a $\phi$-classical prime submodule of $M$, then $( N:_RL)$ is a 
$\psi$-prime ideal of $R$ for every submodule $L$ of $M$ that is not contained in $N$ with $(\phi(N):_RL)\subseteq\psi((N:_RL))$. 
\end{theorem}

\begin{proof}
Suppose that $N$ is a $\phi$-classical prime submodule of $M$ and $L$ is a submodule of $M$ such that $L\nsubseteq N$,
and also $(\phi(N):_RL)\subseteq\psi((N:_RL))$. Let $ab\in \left( N:_RL\right)\backslash\psi(( N:_RL))$ for some $a,b\in R$. Then $abL\subseteq N$ and $abL\nsubseteq\phi(N)$. So by Theorem \ref{main2}, $aL\subseteq N$ or $bL\subseteq N$.  Hence $a\in \left( N:_RL\right)$ or $b\in \left( N:_RL\right) $. Consequently $\left( N:_RL\right) $ is a 
$\psi$-prime ideal of $R$.
\end{proof}

\begin{theorem}
Let $R$ be a $um$-ring, $M$ be a multiplication $R$-module and  $N$ be a proper submodule of $M$.  
Suppose that $\phi :S(M)\rightarrow S(M)\cup\{\emptyset \}$ is a function.
Then the following conditions are equivalent:
\begin{enumerate}
\item[(1)] $N$ is a $\phi$-classical prime submodule of  $M$;
\item[(2)]  If $N_1N_2N_3\subseteq N$  for some submodules $N_1, N_2, N_3$ of $M$ such that $N_1N_2N_3\not\subseteq \phi(N)$, then either $N_1N_3\subseteq N$  or $N_2N_3\subseteq N$;
\item[(3)]  If $N_1N_2\subseteq N$  for some submodules $N_1, N_2$ of $M$ such that $N_1N_2\not\subseteq \phi(N)$, then either $N_1\subseteq N$  or $N_2\subseteq N$;
\item[(4)] $N$ is a $\phi$-prime submodule of $M$.
\end{enumerate}
\end{theorem}
\begin{proof}
(1)$\Rightarrow$ (2) Let $N_1N_2N_3\subseteq N$  for some submodules $N_1, N_2, N_3$ of $M$ such that $N_1N_2N_3\not\subseteq \phi(N)$. Since $M$ is multiplication, there are ideals $I_1,I_2$ of $R$ such that $N_1=I_1M$ and $N_2=I_2M$. Therefore $N_1N_2N_3=I_1I_2N_3\subseteq N$ and $I_1I_2N_3\nsubseteq\phi(N)$, and so by Theorem \ref{main2},  $I_1N_3\subseteq N$ or $I_2N_3\subseteq N$. Hence, $N_1N_3\subseteq N$  or $N_2N_3\subseteq N$.
\\
(2)$\Rightarrow$(3) It is obvious.
\\
(3)$\Rightarrow$(4)  Let $r\in R$ and $m\in M$ with $rm\in N\backslash \phi(N)$. Thus $rRm\subseteq N$ but $rRm\not\subseteq \phi(N)$.  
By Remark \ref{multi}, it follows that $rMRm\subseteq N$. Therefore $rM\subseteq  N$ or $Rm\subseteq N$. Hence $r\in(N:_RM)$ or $m\in N$.\\
(4)$\Rightarrow$(1) By definition. 
\end{proof}

\begin{corollary}
Let $R$ be a $um$-ring, $M$ be a multiplication $R$-module and $N$ be a proper submodule of $M$. Suppose that $\phi :S(M)\rightarrow S(M)\cup\{\emptyset \}$ and $\psi :\mathfrak{J}(R)\rightarrow\mathfrak{J}(R)\cup\{\emptyset \}$ be two functions
with $(\phi(N):_RM)=\psi((N:_RM))$.
Then $N$ is a $\phi$-prime ($\phi$-classical prime) submodule of $M$ if and only if $( N:_RM)$ is a 
$\psi$-prime ideal of $R$.  
\end{corollary}

\begin{proof}
By Theorem \ref{col}, the ``only if'' part holds. Suppose that $( N:_RM)$ is a 
$\psi$-prime ideal of $R$.  Let $IK\subseteq N$ and $IK\nsubseteq\phi(N)$ for some ideal $I$ of $R$ and some
submodule $K$ of $M$. Since $M$ is multiplication, then there is an ideal 
$J$ of $R$ such that $K=JM$. Hence $IJ\subseteq (N:_{R}M)$ and $IJ\nsubseteq\psi((N:_{R}M))$ which
implies that either $I\subseteq (N:_{R}M)$ or $J\subseteq {(N:_{R}M)}$.
If $I\subseteq (N:_{R}M),$ then we are done. So, suppose
that $J\subseteq{(N:_{R}M)}$. Thus $K=JM\subseteq N$. 
Consequently $N$ is $\phi$-prime.
\end{proof}

\begin{theorem}\label{flat}
Let $R$ be a $um$-ring, $M$ be an $R$-module and $\phi :S(M)\rightarrow S(M)\cup
\{\emptyset \}$ be a function.
Suppose that $N$ is a proper submodule of $M$ such that $F\otimes \phi(N)=\phi(F\otimes N)$.
\end{theorem}
\begin{enumerate}
\item If $F$ is a flat $R$-module and $N$ is a $\phi$-classical prime 
submodule of $M$ such that $F\otimes N\neq F\otimes M$, then $%
F\otimes N$ is a $\phi$-classical prime submodule of $F\otimes M.$

\item Suppose that $F$ is a faithfully flat $R$-module. Then $N$ is a
$\phi$-classical prime submodule of $M$ if and only if $F\otimes N$ is a
$\phi$-classical prime submodule of $F\otimes M.$
\end{enumerate}
\begin{proof}
$\left( 1\right) $ Let $a,b\in R$. Then by Theorem \ref{main2}, either $\left(
N:_Mab\right) =\left(\phi(N):_Mab\right) $ or $\left( N:_Mab\right) =\left(
N:_Ma\right) $ or $\left( N:_Mab\right) =\left( N:_Mb\right) $. Assume that $
\left( N:_Mab\right) =\left(\phi(N):_Mab\right) $. Then by \cite[Lemma 3.2]{A}, 
$$\left( F\otimes N:_{F\otimes M}ab\right) =F\otimes \left( N:_Mab\right) =F\otimes \left( \phi(N):_Mab\right)
$$$$\hspace{4.3cm}=\left( F\otimes \phi(N):_{F\otimes M}ab\right)=\left( \phi(F\otimes N):_{F\otimes M}ab\right).$$
Now, suppose that $\left( N:_Mab\right) =\left(N:_Ma\right)$. Again by \cite[Lemma 3.2]{A}, 
$$\left( F\otimes N:_{F\otimes M}ab\right) =F\otimes \left( N:_Mab\right) =F\otimes \left( N:_Ma\right)
$$$$\hspace{0.7cm}=\left( F\otimes N:_{F\otimes M}a\right).$$
Similarly, we can show that if $\left( N:_Mab\right) =\left( N:_Mb\right) $,
then $\left( F\otimes N:_{F\otimes M}ab\right)=\left( F\otimes N:_{F\otimes M}b\right).$
Consequently by Theorem \ref{main2} we deduce that $F\otimes N$ is a
$\phi$-classical prime submodule of $F\otimes M.$

$\left( 2\right) $ Let $N$ be a $\phi$-classical prime submodule of $M$
and assume that $F\otimes N=F\otimes M$. Then $0\rightarrow F\otimes N%
\overset{\subseteq }{\rightarrow }F\otimes M\rightarrow 0$ is an exact
sequence. Since $F$ is a faithfully flat module, $0\rightarrow N\overset{%
\subseteq }{\rightarrow }M\rightarrow 0$ is an exact sequence. So $N=M$,
which is a contradiction. So $F\otimes N\neq F\otimes M$. Then $F\otimes N$
is a $\phi$-classical prime submodule by $\left( 1\right) $. Now for the
converse, let $F\otimes N$ be a $\phi$-classical prime submodule of $%
F\otimes M$. We have $F\otimes N\neq F\otimes M$ and so $N\neq M$. Let $%
a,b\in R$. Then $\left( F\otimes N:_{F\otimes M}ab\right) =\left(\phi(F\otimes N):_{F\otimes M}ab\right) $  
or $\left( F\otimes N:_{F\otimes M}ab\right) =\left( F\otimes
N:_{F\otimes M}a\right) $ or $\left( F\otimes N:_{F\otimes M}ab\right) =\left( F\otimes N:_{F\otimes M}b\right) $
by Theorem \ref{main2}. Suppose $\left( F\otimes N:_{F\otimes M}ab\right) \\=\left(\phi(F\otimes N):_{F\otimes M}ab\right)$. Hence 
$$F\otimes \left( N:_Mab\right)=\left( F\otimes N:_{F\otimes M}ab\right)=\left(
\phi(F\otimes  N):_{F\otimes M}ab\right)
$$$$\hspace{2.5cm}=\left( F\otimes \phi(N):_{F\otimes M}ab\right)=F\otimes \left( \phi(N):_Mab\right).$$
Thus  $0\rightarrow F\otimes \left(\phi(N):_Mab\right) \overset{\subseteq }{%
\rightarrow }F\otimes \left( N:_Mab\right) \rightarrow 0$ is an exact
sequence. Since $F$ is a faithfully flat module, $0\rightarrow \left(
\phi(N):_Mab\right) \overset{\subseteq }{\rightarrow }\left( N:_Mab\right)
\rightarrow 0$ is an exact sequence which implies that $\left( N:_Mab\right)
=\left(\phi(N):_Mab\right) $. With a similar argument we can deduce that if 
$\left( F\otimes N:_{F\otimes M}ab\right) =\left( F\otimes N:_{F\otimes M}a\right) $ or 
$\left( F\otimes N:_{F\otimes M}ab\right) =\left( F\otimes N:_{F\otimes M}b\right)$, then 
$\left( N:_Mab\right)=\left( N:_Ma\right)$ or $\left( N:_Mab\right)=\left( N:_Mb\right) $.
Consequently $N$ is a $\phi$-classical prime
submodule of $M$, by Theorem \ref{main2}.
\end{proof}

\begin{corollary}
Let $R$ be a $um$-ring, $M$ be an $R$-module and $X$ be an indeterminate.
If $N$ is a $\phi$-classical prime submodule of $M$ with $R[X]\otimes \phi(N)=\phi(R[X]\otimes N)$,
then $N[X]$ is a $\phi$-classical prime submodule of $M[X]$.
\end{corollary}
\begin{proof}
Assume that $N$ is a $\phi$-classical prime submodule of $M$ with $R[X]\otimes \phi(N)=\phi(R[X]\otimes N)$. Notice that $R[X]$ is a flat $R$-module. Then by Theorem \ref{flat}, $R[X]\otimes N\simeq N[X]$ is a $\phi$-classical prime submodule of $R[X]\otimes M\simeq M[X]$.
\end{proof}

\begin{definition}
Let $N$ be a proper submodule of $M$ and $a,b\in R$, $m\in M.$
If $N$ is a $\phi$-classical prime submodule and $abm\in\phi(N),$ $am\notin N$, $bm\notin N$, then $(a,b,m)$ is called a 
{\it $\phi$-classical triple-zero of} $N$.
\end{definition}

\begin{theorem}\label{le1}Let $N$ be a $\phi$-classical prime submodule of an $R$-module $M$ and
suppose that $abK\subseteq N$ for some $a,b\in R$ and some submodule $K$ of 
$M$. If $(a,b,k)$ is not a $\phi$-classical triple-zero of $N$ for every $k\in K$, then $aK\subseteq N$ or
$bK\subseteq N$.
\end{theorem}

\begin{proof}
Suppose that $(a,b,k)$ is not a $\phi$-classical triple-zero of $N$ for every $k\in K$. 
Assume on the contrary that $aK\not\subseteq N$ and $bK\not\subseteq N$. Then there are 
$k_{1},k_{2}\in K$ such that $ak_{1}\not\in N$ and $bk_{2}\not\in N$. 
If $abk_{1}\notin\phi(N)$, then we have $bk_{1}\in N$, because 
$ak_{1}\not\in N$ and $N$ is a $\phi$-classical prime submodule
of $M$. If $abk_{1}\in\phi(N)$, then since $ak_{1}\notin N$ and $(a,b,k_{1})$ 
is not a $\phi$-classical triple-zero of $N$, we conclude again that $bk_{1}\in N$. By a
similar argument, since $(a,b,k_2)$ is not a $\phi$-classical triple-zero and $bk_2\notin N$, then we deduce that
$ak_{2}\in N$. By our hypothesis, $ab(k_{1}+k_{2})\in N$
and $(a,b,k_{1}+k_{2})$ is not a $\phi$-classical triple-zero of $N$.
Hence we have either $a(k_{1}+k_{2})\in N$ or $b(k_{1}+k_{2})\in N$. 
If $a(k_{1}+k_{2})=ak_{1}+ak_{2}\in N$, then since $%
ak_{2}\in N$, we have $ak_{1}\in N$, a contradiction. If 
$b(k_{1}+k_{2})=bk_{1}+bk_{2}\in N$, then since $bk_{1}\in N$,
we have $bk_{2}\in N$, which again is a contradiction. Thus 
$aK\subseteq N$ or $bK\subseteq N$.
\end{proof}

\begin{definition}
Let $N$ be a $\phi$-classical prime submodule of an $R$-module $M$ and suppose that 
$IJK\subseteq N$ for some ideals $I,$ $J$ of $R$ and some submodule $K$ of 
$M$. We say that $N$ is a \textit{free }\textit{$\phi$-classical triple-zero with respect to 
}$IJK$ if $(a,b,k)$ is not a $\phi$-classical triple-zero of $N$ for every $a\in
I,b\in J$ and $k\in K$.
\end{definition}

\begin{remark}
Let $N$ be a $\phi$-classical prime submodule of $M$ and suppose that $%
IJK\subseteq N$ for some ideals $I,J$ of $R$ and some submodule $K$ of $M$
such that $N$ is a free $\phi$-classical triple-zero with respect to $IJK$. Hence, if $%
a\in I,b\in J$ and $k\in K$, then $ak\in N$ or $bk\in N$.
\end{remark}

\begin{corollary}
Let $N$ be a $\phi$-classical prime submodule of an $R$-module $M$ and
suppose that $IJK\subseteq N$ for some ideals $%
I,J$ of $R$ and some submodule $K$ of $M$. If $N$ is a free $\phi$-classical 
triple-zero with respect to $IJK$, then $IK\subseteq N$ or $JK\subseteq N$.
\end{corollary}

\begin{proof}
Suppose that $N$ is a free $\phi$-classical triple-zero with respect to $IJK$. 
Assume that $IK\not\subseteq N$ and 
$JK\not\subseteq N$. Then there are $a\in I$ and $b\in J$
with $aK\not\subseteq N$ and $bK\not\subseteq N$.
Since $abK\subseteq N$ and $N$ is free $\phi$-classical 
triple-zero with respect to $IJK$, then Theorem \ref{le1} implies that $aK\subseteq N$ and $bK\subseteq N$
which is a contradiction. Consequently $IK\subseteq N$ or $JK\subseteq N$.
\end{proof}

\begin{theorem}
Let $M$ be an $R$-module and $a$ be an element of $R$ such
that $aM\neq M$. Suppose that $(0:_{M}a)\subseteq aM$. Then $aM$ is an almost 
classical prime submodule of $M$ if and only if it is a classical prime submodule of $M$.
\end{theorem}

\begin{proof}
Assume that $aM$ is an almost classical prime submodule of $M$.
Let $x,y\in R$ and $m\in M$ such that $xym\in aM$. We show that $xm\in aM$ 
or $ym\in aM$. If $xym\notin (aM:_RM)aM$, then there is nothing
to prove, since $aM$ is almost classical prime. So, suppose that $xym\in (aM:_RM)aM$. 
Note that $(x+a)ym\in aM$. If $(x+a)ym\notin (aM:_RM)aM$, then $(x+a)m\in aM$ or 
$ym\in aM$. Hence $xm\in aM$
or $ym\in aM$. Therefore assume that $(x+a)ym\in(aM:_RM)aM$. Hence $xym\in (aM:_RM)aM$ 
gives $aym\in (aM:_RM)aM$. Then, there exists $m^{\prime }\in (aM:_RM)M$ such that $aym=am^{\prime }$ and so $%
ym-m^{\prime }\in (0:_{M}a)\subseteq aM$ which shows that $ym\in aM$, because $m^{\prime}\in aM$. Consequently $aM$ is classical prime. The converse is easy to check.
\end{proof}

\begin{theorem}
Let $M$ be an $R$-module and $m_0$ be an element of $M$ such
that $Rm_0\neq M$. Suppose that $(0:_{R}m_0)\subseteq(Rm_0:_RM)$. If $Rm_0$ is an almost 
classical prime submodule of $M$, then it is a 2-absorbing submodule of $M$.
\end{theorem}

\begin{proof}
Assume that $Rm_0$ is an almost classical prime submodule of $M$.
Let $x,y\in R$ and $m\in M$ such that $xym\in Rm_0$. If $xym\notin (Rm_0:_RM)m_0$, then there is nothing
to prove, since $Rm_0$ is almost classical prime. So, suppose that $xym\in (Rm_0:_RM)m_0$. 
Notice that $xy(m+m_0)\in Rm_0$. If $xy(m+m_0)\notin (Rm_0:_RM)m_0$, then $x(m+m_0)\in Rm_0$ or 
$y(m+m_0)\in Rm_0$. Hence $xm\in Rm_0$ or $ym\in Rm_0$. Therefore assume that $xy(m+m_0)\in(Rm_0:_RM)m_0$. 
Hence $xym\in (Rm_0:_RM)m_0$ implies that $xym_0\in (Rm_0:_RM)m_0$. Then, there exists $r\in (Rm_0:_RM)$ such that $xym_0=rm_0$ and so $xy-r\in (0:_{R}m_0)\subseteq(Rm_0:_RM)$ which shows that $xy\in(Rm_0:_RM)$. Consequently $Rm_0$ is a 2-absorbing
submodule of $M$.
\end{proof}

\begin{proposition}
Let $N$ be a submodule of $M$ and $\phi (N)$ be a classical prime
submodule of $M$. If $N$ is a $\phi $-classical prime submodule of $M$,
then $N$ is a classical prime submodule of $M$.
\end{proposition}

\begin{proof}
Let $N$ be a $\phi $-classical prime submodule of $M$. Assume that $%
abm\in N$ for some elements $a,b\in R$ and $m\in M$. If $abm\in \phi (N)$,
then since $\phi (N)$ is classical prime, we conclude that $am\in\phi (N)\subseteq N$ 
or $bm\in\phi (N)\subseteq N$, and so we are done. When $abm\notin \phi (N)$ clearly
the result follows.
\end{proof}

Let $\mathcal{S}$ be a multiplicatively closed subset of $R$. It is well-known that
each submodule of $\mathcal{S}^{-1}M$ is in the form of $\mathcal{S}^{-1}N$ for some submodule $N$
of $M$. Let $\phi :S(M)\rightarrow S(M)\cup \{\emptyset \}$ be a function and define
$\phi _{\mathcal{S}}:S(\mathcal{S}^{-1}M)\rightarrow S(\mathcal{S}^{-1}M)\cup
\{\emptyset\}$ by $\phi_{\mathcal{S}}(\mathcal{S}^{-1}N)=\mathcal{S}^{-1}\phi(N)$ 
(and $\phi_{\mathcal{S}}(\mathcal{S}^{-1}N)=\emptyset$ when $\phi(N)=\emptyset$)
for every submodule $N$ of $M$. 

For an $R$-module $M$, the set of zero-divisors of $M$ is denoted by $Z_R(M)$.
\begin{theorem}
Let $M$ be an $R$-module, $N$ be a submodule and $\mathcal{S}$ be a multiplicative
subset of $R$. 
\begin{enumerate}
\item If $N$ is a $\phi$-classical prime submodule of $M$ such that $\left( N:_RM\right)\cap \mathcal{S}=\emptyset$,
then $\mathcal{S}^{-1}N$ is a $\phi_{\mathcal{S}}$-classical prime submodule of $\mathcal{S}^{-1}M$.
\item If $\mathcal{S}^{-1}N$ is a $\phi_\mathcal{S}$-classical prime submodule of $\mathcal{S}^{-1}M$ such that  $\mathcal{S}\cap Z_R(N/\phi(N))\\=\emptyset$ and $\mathcal{S}\cap Z_R(M/N)=\emptyset$, then $N$ is a $\phi$-classical prime submodule of $M$.
\end{enumerate}
\end{theorem}

\begin{proof}
(1) Let $N$ be a $\phi$-classical prime submodule of $M$ and 
$\left( N:_RM\right) \cap \mathcal{S}=\emptyset$. Suppose that 
$\frac{a_{1}}{s_{1}}\frac{a_{2}}{s_{2}}\frac{m}{s_{3}}\in \mathcal{S}^{-1}N\backslash\phi_{\mathcal{S}}(\mathcal{S}^{-1}N)$
for some $a_1,a_2\in R$, $s_1,s_2,s_3\in \mathcal{S}$ and $m\in M$.
Then there exists $s\in \mathcal{S}$ such that $sa_1a_2m\in N$. 
If $sa_1a_2m\in\phi(N)$, then $\frac{a_{1}}{s_{1}}\frac{a_{2}}{s_{2}}\frac{m}
{s_{3}}=\frac{sa_1a_2m}{ss_1s_2s_3}\in\phi_{\mathcal{S}}(\mathcal{S}^{-1}N)=\mathcal{S}^{-1}\phi(N)$, a contradiction.
Since $N$ is a $\phi$-classical prime submodule, then we have $a_{1}\left(
sm\right)\in N$ or $a_{2}\left( sm\right) \in N$. Thus 
$\frac{a_{1}}{s_{1}}\frac{m}{s_{3}}=\frac{sa_{1}m}{ss_{1}s_3}\in \mathcal{S}^{-1}N$ or
$\frac{a_{2}}{s_{2}}\frac{m}{s_{3}}=\frac{sa_{2}m}{ss_{2}s_3}\in \mathcal{S}^{-1}N$.
Consequently $\mathcal{S}^{-1}N$ is a $\phi_{\mathcal{S}}$-classical prime submodule of $\mathcal{S}^{-1}M$.

(2) Suppose that $\mathcal{S}^{-1}N$ is a $\phi_{\mathcal{S}}$-classical prime submodule of $\mathcal{S}^{-1}M$, \ $\mathcal{S}\cap Z_R(N/\phi(N))=\emptyset$ and $\mathcal{S}\cap Z_R(M/N)=\emptyset$. Let $a,b\in R$ and $m\in M$ such that $abm\in N\backslash\phi(N)$. Then $\frac{a}{1}\frac{b}{1}\frac{m}{1}\in \mathcal{S}^{-1}N$. If 
$\frac{a}{1}\frac{b}{1}\frac{m}{1}\in\phi_{\mathcal{S}}(\mathcal{S}^{-1}N)=\mathcal{S}^{-1}\phi(N)$, 
then there exists $s\in S$ such that $sabm\in\phi(N)$
which contradicts $\mathcal{S}\cap Z_R(N/\phi(N))=\emptyset$. 
Therefore $\frac{a}{1}\frac{b}{1}\frac{m}{1}\in\mathcal{S}^{-1}N\backslash\phi_{\mathcal{S}}(\mathcal{S}^{-1}N)$, and so either $\frac{a}{1}\frac{m}{1}\in \mathcal{S}^{-1}N$ or $\frac{b}{1}\frac{m}{1}\in \mathcal{S}^{-1}N$. We may assume that $\frac{a}{1}\frac{m}{1}\in \mathcal{S}^{-1}N$. So there exists $u\in \mathcal{S}$ such that $uam\in N$. But $\mathcal{S}\cap Z_R(M/N)=\emptyset$, whence $am\in N$. Consequently $N$ is a $\phi$-classical prime submodule of $M$.
\end{proof}

\begin{theorem}\label{T1} 
Let $N$ be a $\phi$-classical prime submodule of $M$ and suppose that $(a,b,m)$ 
is a $\phi$-classical triple-zero of $N$ for some $a,b\in R,$ $m\in M$. Then

\begin{enumerate}
\item $abN\subseteq\phi(N)$.

\item $a(N:_RM)m\subseteq\phi(N)$.

\item $b(N:_RM)m\subseteq\phi(N)$.

\item $(N:_RM)^2m\subseteq\phi(N)$.

\item $a(N:_RM)N\subseteq\phi(N)$.

\item $b(N:_RM)N\subseteq\phi(N)$.
\end{enumerate}
\end{theorem}

\begin{proof}
(1) Suppose that $abN\nsubseteq\phi(N)$. Then there exists $n\in N$
with $abn\notin\phi(N)$. Hence $ab(m+n)\in N\backslash\phi(N)$, so we conclude that $a(m+n)\in
N $ or $b(m+n)\in N$. Thus $am\in N$ or $bm\in N$, which contradicts the assumption that $(a,b,m)$ is $\phi$-classical triple-zero.
Thus $abN\subseteq\phi(N)$.

(2) Let $axm\notin\phi(N)$ for some $x\in (N:_{R}M)$. Then $a(b+x)m\notin\phi(N)$, because $abm\in\phi(N)$. 
Since $xm\in N$, then $a(b+x)m\in N$. Then $am\in N$ or $(b+x)m\in N$. Hence 
$am\in N$ or $bm\in N$, which contradicts our hypothesis. 

(3) The proof is similar to part (2).

(4) Assume that $x_{1}x_{2}m\notin\phi(N)$ for some $x_{1},x_{2}\in
(N:_{R}M)$. Then by parts (2) and (3), $(a+x_{1})(b+x_{2})m\notin\phi(N)$.
Clearly $(a+x_{1})(b+x_{2})m\in N$. Then $(a+x_{1})m\in N$ or $(b+x_{2})m\in
N$. Therefore $am\in N$ or $bm\in N$ which is a contradiction. Consequently $(N:_{R}M)^{2}m\subseteq\phi(N).$

(5) Let $axn\notin\phi(N)$ for some $x\in(N:_RM)$ and $n\in N$. Therefore by parts (1) and (2) we conclude that
$a(b+x)(m+n)\in N\backslash\phi(N)$. So $a(m+n)\in N$ or $(b+x)(m+n)\in N$. Hence
$am\in N$ or $bm\in N$. This contradiction shows that $a(N:_RM)N\subseteq\phi(N).$

(6) Similart to part (5).
\end{proof}

\begin{theorem}\label{T2} 
If $N$ is a $\phi$-classical prime submodule of an $R$-module $M$ that is
not classical prime, then $(N:_{R}M)^{2}N\subseteq\phi(N)$.
\end{theorem}

\begin{proof}
Suppose that $N$ is a $\phi$-classical prime submodule of $M$ that is
not classical prime. Then there exists a $\phi$-classical triple-zero $%
(a,b,m)$ of $N$ for some $a,b\in R$ and $m\in M$. Assume that $%
(N:_{R}M)^{2}N\nsubseteq\phi(N)$. Hence there are $x_{1},x_{2}\in
(N:_{R}M) $ and $n\in N$ such that $x_{1}x_{2}n\notin\phi(N)$. By Theorem %
\ref{T1}, $(a+x_{1})(b+x_{2})(m+n)\in N\backslash\phi(N)$. So $%
(a+x_{1})(m+n)\in N$ or $(b+x_{1})(m+n)\in N$. Therefore $am\in N$ 
or $bm\in N$, a contradiction.
\end{proof}

\begin{corollary} 
Let $M$ be an $R$-module and $N$ be a $\phi$-classical prime submodule of M such that $\phi(N)\subseteq(N:_{R}M)^3N$. 
Then $N$ is $\omega$-classical prime.
\end{corollary} 

\begin{proof}
If $N$ is a classical prime submodule of $M$, then it is clear. Hence, suppose that $N$ is not a classical prime submodule of $M$.
Therefore by Theorem \ref{T2} we have $(N:_{R}{M})^2N\subseteq\phi(N)\subseteq(N:_{R}M)^3N\subseteq(N:_{R}M)^2N$, that is, $\phi(N)=(N:_{R}M)^2N=(N:_{R}M)^3N$. Therefore, $\phi(N)=(N:_{R}M)^jN$ for all $j\geq 2$ and the result is obtained.
\end{proof}

As a direct consequence of Theorem \ref{T2} we have the following result.
\begin{corollary}\label{almost}
Let $M$ be an $R$-module and $N$ be a proper submodule of $M$.
If $N$ is an $n$-almost classical prime submodule ($n\geq3$) of $M$ that is
not classical prime, then $(N:_{R}M)^{2}N=(N:_{R}M)^{n-1}N$.
\end{corollary}

\begin{corollary}
Let $M$ be a multiplication $R$-module and $N$ be a proper submodule of $M$.
\begin{enumerate}
\item If $N$ is a $\phi$-classical prime submodule of $M$ that is not classical prime, then $N^{3}\subseteq\phi(N)$.
\item If $N$ is an $n$-almost classical prime submodule ($n\geq3$) of $M$ that is not classical prime, then $N^{3}=N^n$.
\end{enumerate}
\end{corollary}
\begin{proof}
(1) Since $M$ is multiplication, then $N=(N:_RM)M$. Therefore by Theorem \ref{T2} and Remark \ref{multi}, $N^3=(N:_RM)^2N\subseteq\phi(N)$.

(2) Notice that $\phi_n(N)=(N:_{R}M)^{n-1}N=N^n$. Now, use part (1).
\end{proof}

\begin{theorem}\label{nil} 
Let $N$ be a $\phi$-classical prime submodule of $M$. If $N$
is not classical prime, then

\begin{enumerate}
\item $\sqrt{(N:_RM)}=\sqrt{(\phi(N):_RM)}$.

\item If $M$ is multiplication, then $M$-${rad}(N)$=$M$-${rad}(\phi(N))$.
%If in addition $M$ is faithful, then $M$-$\mathrm{rad}(N)=\mathrm{Nil}(M)$.
\end{enumerate}
\end{theorem}

\begin{proof}
(1) Assume that $N$ is not classical prime. By Theorem \ref{T2}, $%
(N:_RM)^2N\subseteq\phi(N)$. Then 
\begin{equation*}
(N:_RM)^3=(N:_RM)^2(N:_RM)
\end{equation*}
\begin{equation*}
\hspace{2cm}\subseteq((N:_RM)^2N:_RM)
\end{equation*}
\begin{equation*}
\hspace{.8cm}\subseteq(\phi(N):_RM),
\end{equation*}
and so $(N:_RM)\subseteq\sqrt{(\phi(N):_RM)}$. Hence, we have $\sqrt{(N:_RM)}=
\sqrt{(\phi(N):_RM)}$.

(2) By part (1), $M $-$\mathrm{rad}(N)=\sqrt{(N:_{R}M)}M=\sqrt{(\phi(N):_RM)}M=M$-$\mathrm{rad}(\phi(N))$.
\end{proof}

\begin{theorem}
Let $M$ be an $R$-module. Suppose that $N_1,N_2$ are $\phi$-classical prime submodules of $M$ that are not 
classical prime submodules. Then
\begin{enumerate}
\item $\sqrt{(N_1:_RM)+(N_2:_RM)}=\sqrt{(\phi(N_1):_RM)+(\phi(N_2):_RM)}$.
\item If $N_1+N_2\neq M$, $\phi(N_1)\subseteq N_2$ and $\phi(N_2)\subseteq\phi(N_1+N_2)$,
then $N_1+N_2$ is a $\phi$-classical prime submodule.
\end{enumerate}
\end{theorem}
\begin{proof}
(1) By Theorem \ref{nil}, we have $\sqrt{(N_1:_RM)}=\sqrt{(\phi(N_1):_RM)}$  and 
$\sqrt{(N_2:_RM)}\\=\sqrt{(\phi(N_2):_RM)}$. Now, by \cite[2.25(i)]{Sh} the result follows.

(2) Suppose that $N_1+N_2\neq M$, $\phi(N_1)\subseteq N_2$ and $\phi(N_2)\subseteq\phi(N_1+N_2)$.
Since $(N_1+N_2)/N_2\simeq N_1/(N_1\cap N_2)$ and $N_1$ is $\phi$-classical prime, we get
 $(N_1+N_2)/N_2$ is a weakly classical prime submodule of $M/N_2$, by Theorem \ref{frac}(3).
Now, the assertion follows from Theorem \ref{frac}(4).
\end{proof}

\begin{theorem}\label{prod1}
Let $M_{1}, M_{2}$ be $R$-modules and
$N_{1}$ be a proper submodule of $M_1$. 
Suppose that $\psi _{i}:S(M_{i})\rightarrow S(M_{i})\cup\{\emptyset \}$ be functions (for $i=1,2$) and let $\phi =\psi _{1}\times\psi _{2}$. Then the following conditions are equivalent:
\begin{enumerate}
\item $N=N_{1}\times M_{2}$ is a $\phi$-classical prime submodule of $M=M_{1}\times M_{2}$;

\item $N_1$ is a $\phi$-classical prime submodule of $M_1$ and for each $r,s\in R$ and $m_1\in M_1$
we have $$rsm_1\in\psi_1(N_1),\ rm_1\notin N_1, \ sm_1\notin N_1\Rightarrow rs\in (\psi_2(M_2):_RM_2).$$
\end{enumerate}
\end{theorem}
\begin{proof}
(1)$\Rightarrow$(2) Suppose that $N=N_{1}\times M_{2}$ is a $\phi$-classical prime submodule of 
$M=M_{1}\times M_{2}$. Let $r,s\in R$ and $m_1\in M_1$ be such that $rsm_1\in N_1\backslash\psi_1(N_1)$.
Then $rs(m_1,0)\in N\backslash\phi(N)$. Thus $r(m_1,0)\in N$ or $s(m_1,0)\in N$, and so $rm_1\in N_1$ or $sm_1\in N_1$.
Consequently $N_1$ is a $\phi$-classical prime submodule of $M_1$. Now, assume that $rsm_1\in\psi(N_1)$
for some $r,s\in R$ and $m_1\in M_1$ such that $rm_1\notin N_1$ and $sm_1\notin N_1$.
Suppose that $rs\notin(\psi_2(M_2):_RM_2)$. Therefore there exists $m_2\in M_2$ such that $rsm_2\notin\psi_2(M_2)$.
Hence $rs(m_1,m_2)\in N\backslash\phi(N)$, and so $r(m_1,m_2)\in N$ and $s(m_1,m_2)\in N$.
Thus $rm_1\in N_1$ or $sm_1\in N_1$ which is a contradiction. Consequently $rs\in(\psi_2(M_2):_RM_2)$.\newline
(2)$\Rightarrow$(1) Let $r,s\in R$ and $(m_1,m_2)\in M=M_{1}\times M_{2}$ be such that 
$rs(m_1,m_2)\in N\backslash\phi(N)$. First assume that $rsm_1\notin\psi_1(N_1)$. Then by part (2),
$rm_1\in N_1$ or $sm_1\in N_1$. So $r(m_1,m_2)\in N$ or $s(m_1,m_2)\in N$, and thus we are done.
If $rsm_1\in\psi_1(N_1)$, then $rsm_2\notin\psi_2(M_2)$. Therefore $rs\notin(\psi_2(M_2):_RM_2)$, and so part (2) implies that
either $rm_1\in N_1$ or $sm_1\in N_1$. Again we have that $r(m_1,m_2)\in N$ or $s(m_1,m_2)\in N$ which 
shows $N$ is a $\phi$-classical prime submodule of $M$.
\end{proof}

\begin{proposition}\label{p1}
Let $R=R_1\times R_2$ be a decomposable ring, $M_1$ be an $R_1$-module and $M_2$ be an $R_2$-module. Suppose that $\phi:S(M)\rightarrow S(M)\cup\{\emptyset\}$ is a function and $M=M_1\times M_2$. If $N_1$ is a weakly classical prime submodule of $M_1$ satisfying $\{0\}\times M_2\subseteq\phi(N_1\times M_2)$, then $N_1\times M_2$ is a $\phi$-classical prime submodule of $M$.
\end{proposition}

\begin{proof}
Let $(a,b),(c,d)\in R$ and $(m_1,m_2)\in M$ such that $(a,b)(c,d)(m_1,m_2)\in N_1\times M_2\backslash\phi(N_1\times M_2)$. Then $0\neq acm_1\in N_1$. So $am_1\in N_1$ or $cm_1\in N_1$, since $N_1$ is a weakly classical prime submodule. Hence, $(a,b)(m_1,m_2)\in N_1\times M_2$ or $(c,d)(m_1,m_2)\in N_1\times M_2$.
\end{proof}

\begin{corollary}
Let $R=R_{1}\times R_{2}$ be a decomposable ring and $M=M_{1}\times R_{2}$ be an $R$-module where
$M_1$ is an $R_1$-module. 
If $N_1$ is a weakly classical prime submodule of $M_1$, then $N=N_1\times R_2$ is a 3-almost classical prime submodule of $M$.
\end{corollary}
\begin{proof}
Suppose that $N_1$ is a weakly classical prime submodule of $M_1$. 
If $N_1$ is a classical prime submodule of $M_1$, then it is easy to see that
$N$ is a classical prime submodule of $M$ and so is $\phi$-classical prime submodule of $M$, for all $\phi$.
Assume that $N_1$ is not classical prime. Therefore by \ref{T2}, $(N_1:_{R_1}M_1)^2N_1=\{0\}$
and so $\phi_3(N)=(N:_RM)^2N=\{0\}\times R_2$. Now, by Proposition \ref{p1} the result follows.
%Let $(a_1,a_2)(b_1,b_2)(m,r)\in N\backslash\phi_3(N)$
%for some $(a_1,a_2),(b_1,b_2)\in R$ and $(m,r)\in M$. Hence $0\neq a_1b_1m\in N_1$. Then $a_1m\in N_1$
%or $b_1m\in N_1$. Thus $(a_1,a_2)(m,r)\in N$ or $(b_1,b_2)(m,r)\in N$. Consequently 
%$N$ is a 3-almost classical prime submodule of $M$.
\end{proof}

\begin{theorem}
Let  $R=R_1\times R_2$ be a decomposable ring and $M=M_1\times M_2$ be an $R$-module where $M_1$ is an $R_1$-module 
and $M_2$ is an $R_2$-module. Let 
$\psi _i:S\left( M_i\right)
\rightarrow S(M_i)\cup \left\{ \varnothing \right\} $ be  functions for $i=1, 2$ and let  $\phi= \psi_1 \times \psi_2$. 
If $N=N_1\times M_2$ is a proper submodule of $M$, then the following conditions are equivalent:

\begin{itemize}
\item[(1)] $N_1$ is a classical prime submodule of $M_1$;
\item[(2)]   $N$ is a classical prime submodule of $M$;
\item[(3)]   $N$ is $\phi$-classical prime submodule of $M$  where $\psi_2(M_2)\neq M_2$.
\end{itemize}

\end{theorem}

\begin{proof}
(1)$\Rightarrow$(2) Let $(a_1,a_2)(b_1,b_2)(m_1,m_2)\in N$ for some $(a_1,a_2),(b_1,b_2)\in R$ and $(m_1,m_2)\in M$. 
Then $a_1b_1m_1\in N_1$ so either $a_1m_1\in N_1$ or $b_1m_1\in N_1$ which shows that either $(a_1,a_2)(m_1,m_2)\in N$
or $(b_1,b_2)(m_1,m_2)\in N$. Consequently $N$ is a classical prime submodule of $M$.\newline
(2)$\Rightarrow$(3) It is clear that every classical prime submodule is a weakly classical prime submodule.\newline
(3)$\Rightarrow$(1) Let $abm\in N_1$ for some $a, b\in R_1$  and $m\in M_1$. 
By assumption, there exists  $m'\in M_2\backslash\psi_2(M_2)$. Thus $(a, 1)(b,1)(m,m')\in N\backslash\phi(N)$. So we have  
$(a, 1)(m,m')\in N$ or $(b,1)(m,m')\in N$. Thus $am\in N_1$ or $bm\in N_1$. Therefore $N_1$ is a classical prime submodule of $M_1$.
\end{proof}

\begin{theorem}
Let  $R=R_1\times R_2$ be a decomposable ring and $M=M_1\times M_2$ be an $R$-module where $M_1$ is an $R_1$-module 
and $M_2$ is an $R_2$-module. Let $\psi _i:S\left( M_i\right)\rightarrow S(M_i)\cup \left\{ \varnothing \right\}$ be functions 
for $i=1, 2$ where $\psi_2(M_2)=M_2$, and let  $\phi= \psi_1 \times \psi_2$. 
If $N=N_1\times M_2$ is a proper submodule of $M$, then 
$N_1$ is a $\psi_1$-classical prime submodule of $M_1$ if and only if
$N$ is $\phi$-classical prime submodule of $M$.
\end{theorem}

\begin{proof}
Suppose that $N$ is a $\phi$-classical prime submodule of $M$. First we
show that $N_{1}$ is a $\psi _{1}$-classical prime submodule of $M_{1}$
independently whether $\psi _{2}(M_{2})=M_{2}$ or $\psi
_{2}(M_{2})\not=M_{2}.$ Let $a_{1}b_{1}m_{1}\in N_{1}\backslash
\psi _{1}(N_{1})$ for some $a_{1},b_{1}\in R_{1}$ and $m_{1}\in M_{1}.$
Then $(a_{1},1)(b_{1},1)(m_{1},m)\in (N_{1}\times M_{2})\backslash (\psi
_{1}(N_{1})\times \psi _{2}(M_{2}))=N\backslash \phi (N)$ for any $m\in
M_{2}.$ Since $N$ is a $\phi $-classical prime submodule of $M$, we get either 
$(a_{1},1)(m_{1},m)\in N_{1}\times M_{2}$ or $(b_{1},1)(m_{1},m)\in
N_{1}\times M_{2}$. So, clearly we conclude that 
$a_{1}m_{1}\in N_{1}$ or $b_{1}m_{1}\in N_{1}.$ Therefore $
N_{1}$ is a $\psi _{1}$-classical prime submodule of $M_{1}$.
Conversely, suppose that $N_{1}$ is $\psi _{1}$-classical prime. 
Let $(a_{1},a_{2})$, $(b_{1},b_{2})\in R$ and $(m_{1},m_{2})\in M$ 
such that $(a_{1},a_{2})(b_{1},b_{2})(m_{1},m_{2})\in N\backslash
\phi (N)$. Since $\psi _{2}(M_{2})=M_{2},$ we get $a_{1}b_{1}m_{1}\in $ $
N_{1}\backslash \psi _{1}(N_{1})$ and this implies that either $a_{1}m_{1}\in N_{1}$ or $b_{1}m_{1}\in
N_{1}$. Thus $(a_{1},a_{2})(m_{1},m_{2})\in N$ or $(b_{1},b_{2})(m_{1},m_{2})\in N$.
\end{proof}

\begin{theorem}
Let  $R=R_1\times R_2$ be a decomposable ring and $M=M_1\times M_2$ be an $R$-module where $M_1$ is an $R_1$-module 
and $M_2$ is an $R_2$-module.  Suppose that $N_1, N_2 $ are proper  submodules of $M_1,M_2$, respectively. 
Let 
$\psi _i:S\left( M_i\right)
\rightarrow S(M_i)\cup \left\{ \varnothing \right\} $ be  functions for $i=1, 2$ and let  $\phi= \psi_1 \times \psi_2$. 
If $N=N_1\times N_2$ is a $\phi$-classical prime  submodule of $M$, then $N_1$ is a  $\psi_1$-classical prime submodule of $M_1$
and $N_2$ is a    $\psi_2$-classical prime submodule of $M_2$.
\end{theorem}

\begin{proof}
Suppose that  $N=N_1\times N_2$ is a $\phi$-classical prime submodule of $M$. Let $abm\in N_1\backslash\psi_1(N_1)$ that $a,b\in R_1$ and $m\in M_1$. Get an element $n\in N_2$. We have $(a,1)(b,1)(m,n)\in N\backslash\phi(N)$. Then $(a,1)(m,n)\in N$ or $(b,1)(m,n)\in N$. Thus $am\in N_1$ or $bm\in N_1$, and thus $N_1$ is a $\psi_1$-classical prime submodule of $M_1$. By a simillar argument we can show that  $N_2$ is a $\psi_2$-classical prime submodule of $M_2$.
\end{proof}

%\begin{corollary}
%Let $R$ be a ring and $I$ be a proper ideal of $R$.
%\begin{enumerate}
%\item  $_{R}I$ is a weakly classical prime submodule of $_{R}R$ if and only if $I$ is a weakly prime ideal of $R$. 

%\item Every proper ideal of $R$ is weakly prime if and only if for every $R$-module $M$
%and every proper submodule $N$ of $M$, $N$ is a weakly classical prime submodule of $M$.
%\end{enumerate} 
%\end{corollary}
%\begin{proof}
%(1) Let $_{R}I$ be a weakly classical prime submodule of $_{R}R$. Then by Theorem \ref{colon}(1),
%$(I:_R1)=I$ is a weakly prime ideal of $R$. For the converse, notice that $_{R}I$ is a weakly prime submodule of $_{R}R$ if and only if $I$ is a weakly prime ideal of $R$. Now, apply part (1) of Proposition \ref{abs-class}.

%(2) Assume that every proper ideal of $R$ is weakly prime.
%Let $N$ be a proper submodule of an $R$-module $M$.
%Since for every $m\in M\backslash N$, $(N:_Rm)$ is a proper ideal of $R$, then it is a weakly prime ideal of $R$. 
%Hence by Theorem \ref{colon}(2), $N$ is a weakly classical prime submodule of $M$. We have the converse immediately 
%by part (1).
%\end{proof}

\begin{theorem}\label{product3}
Let $R=R_{1}\times R_{2}\times R_3$ be a decomposable ring and $M=M_{1}\times M_{2}\times M_3$ be an $R$-module where
$M_{i}$ is an $R_{i}$-module, for $i=1,2,3$. Suppose that $\psi _{i}:S(M_{i})\rightarrow S(M_{i})\cup\{\emptyset \}$ be functions such that $\psi(M_i)\neq M_i$ for i=1,2,3, and let $\phi =\psi _{1}\times\psi _{2}\times\psi_3.$  If $N$ is a $\phi$-classical prime submodule of $M$, then either $N=\phi(N)$ or $N$ is a classical prime submodule of $M$.
\end{theorem}
\begin{proof}
If $N=\phi(N)$, then clearly $N$ is a $\phi$-classical prime submodule of $M$, so we may assume
that $N=N_{1}\times N_{2}\times N_3\neq\psi_1(N_1)\times\psi_2(N_2)\times\psi_3(N_3)$.  
Without loss of generality we may assume that $N_1\neq\psi_1(N_1)$ and so there is $n\in N_1\backslash\psi(N_1)$.
We claim that $N_2=M_2$ or $N_3=M_3$. Suppose that there are $m_2\in M_2\backslash N_2$ and $m_3\in M_3\backslash N_3$.
Get $r\in(N_2:_{R_2}M_2)$ and $s\in(N_3:_{R_3}M_3)$. Since
$(1,r,1)(1,1,s)(n,m_2,m_3)=(n,rm_2,sm_3)\in N\backslash\phi(N)$, then $(1,r,1)(n,m_2,m_3)=(n,rm_2,m_3)\in N$
or $(1,1,s)(n,m_2,m_3)=(n,m_2,sm_3)\in N$. Therefore either $m_3\in N_3$ or $m_2\in N_2$, a contradiction.
Hence $N=N_{1}\times M_{2}\times N_3$ or $N=N_{1}\times N_{2}\times M_3$. Let $N=N_{1}\times M_{2}\times N_3$.
Then $(0,1,0)\in(N:_RM)$. Clearly $(0,1,0)^2N\nsubseteq\psi_1(N_1)\times\psi_2(M_2)\times\psi(N_3)$. 
So $(N:_{R}M)^{2}N\nsubseteq\phi(N)$ which
is a contradiction, by Theorem \ref{T2}. In the case when $N=N_{1}\times N_{2}\times M_3$ we have that
$(0,0,1)\in(N:_RM)$ and similar to the previous case we reach a contradiction.
\end{proof}

\vspace{5mm} \noindent \footnotesize 
\begin{minipage}[b]{10cm}
Hojjat Mostafanasab \\
Department of Mathematics and Applications, \\ 
University of Mohaghegh Ardabili, \\ 
P. O. Box 179, Ardabil, Iran. \\
Email: h.mostafanasab@uma.ac.ir, \hspace{1mm} h.mostafanasab@gmail.com
\end{minipage}\\

\vspace{2mm} \noindent \footnotesize
\begin{minipage}[b]{10cm}
Esra Sengelen Sevim\\
Eski Silahtara\v{g}a Elektrik Santrali, Kazim Karabekir,\\
Istanbul Bilgi University,\\  
Cad. No: 2/1334060, Ey\"{u}p Istanbul, Turkey. \\
Email: esra.sengelen@bilgi.edu.tr\\
\end{minipage}\\

\vspace{5mm} \noindent \footnotesize 
\begin{minipage}[b]{10cm}
Sakineh Babaei \\
Department of Mathematics, \\ 
Imam Khomeini International University, \\ 
Qazvin, Iran. \\
Email: sbabaei@edu.ikiu.ac.ir
\end{minipage}\\

\vspace{2mm} \noindent \footnotesize
\begin{minipage}[b]{10cm}
\"{U}nsal Tekir\\
Department of Mathematics, \\ 
Marmara University, \\ 
Ziverbey, Goztepe, Istanbul 34722, Turkey. \\
Email: utekir@marmara.edu.tr
\end{minipage}\\

\end{document}